\documentclass{amsart}
\usepackage{amssymb}

\usepackage[utf8]{inputenc}
\usepackage[T1]{fontenc}

\usepackage{tikz}
\usetikzlibrary{trees}

\usepackage{color}
\usepackage{url}
\usepackage{enumerate}

\usepackage[normalem]{ulem}

\theoremstyle{plain} 
\newtheorem{theorem}{Theorem}
\newtheorem{lemma}[theorem]{Lemma}

\theoremstyle{remark}
\newtheorem*{remark}{Remark}
\theoremstyle{definition}

\newcommand{\F}{\mathcal{F}}
\newcommand{\bR}{\mathbb{R}}

\newcommand{\bN}{\mathbb{N}}

\newcommand{\cB}{\mathcal{B}}

\newcommand{\cC}{\mathcal{C}}
\newcommand{\cD}{\mathcal{D}}

\newcommand{\FF}{\mathcal{F}}

\newcommand{\cA}{\mathcal{A}}

\newcommand{\Baire}{\mathrm{Baire}}
\newcommand{\diam}{\mathrm{diam}}
\newcommand{\Leb}{\mathrm{Leb}}

\newcommand{\la}{\langle}
\newcommand{\ra}{\rangle}

\title{Games characterizing certain families of functions}

\author{Marek Balcerzak}
\address{Institute of Mathematics,
         Lodz University of Technology, al. Politechniki 8,
         93-590 \L\'od\'z,
         Poland}
         \email{marek.balcerzak@p.lodz.pl}

\author{Tomasz Natkaniec}
\address{Institute of Mathematics, Faculty of Mathematics, Physics and Informatics,
University of Gda\'{n}sk, 80-308 Gda\'{n}sk, Poland}
\email{tomasz.natkaniec@ug.edu.pl}

\author{Piotr Szuca}
\address{Institute of Mathematics, Faculty of Mathematics, Physics and Informatics,
University of Gda\'{n}sk, 80-308 Gda\'{n}sk, Poland}
\email{piotr.szuca@ug.edu.pl}

\begin{document}



\begin{abstract}
We obtain several game characterizations of Baire 1 functions between Polish spaces $X$, $Y$ which extends the recent result of V. Kiss.
Then we propose similar characterizations for equi-Bare 1 families of functions.
Also, using related ideas, we give game characterizations of Baire measurable
and Lebesgue measurable functions.
\end{abstract}

\maketitle

\section{Introduction}
The game approach plays an important role in descriptive set theory. Let us recall Choquet games 
and the Banach-Mazur game in the studies of the Baire category problems \cite[Sec. 8]{Ke}, and
Wadge games with their influence on investigations in the Borel hierarchy \cite[Sec. 21]{Ke}. It is commonly known that
Borel and projective
determinacy provide a strong tool in set-theoretical investigations, cf. \cite[Sec. 20, 38]{Ke}. Note that
various kinds of topological games make fruitful inspirations in topology and analysis, cf. \cite{AD}, \cite{Gr}, \cite{CM}.
They can distinguish new kinds of topological objects, cf. \cite{Gr}, \cite{AD}.

In the recent decades, several nice characterizations for some classes of regular functions were obtained.
Duparc \cite{D} and Carroy \cite{C} characterized Baire 1 functions from $\bN^\bN$ into itself by using the so-called eraser game (for more applications of this game, see \cite{CD}).
Other significant results for different classes of functions between Polish
zero-dimensional spaces are due to
Andretta \cite{An} (a game characterization of ${\pmb \Delta^0_2}$-measurable functions), Semmes \cite{Se}
(Borel functions), Nobrega \cite{No} (Baire class $\xi$ functions)\footnote{We would like to thank the reviewer 
  for noting that this result was announced by Louveau and Semmes at a conference in 2010.}
and Motto Ros \cite{Ro} (piecewise defined functions).

This note is motivated by Kiss \cite{Ki}, who introduced a game characterizing Baire class 1 functions between arbitrary two Polish spaces.
This improved the results by Duparc \cite{D} and Carroy \cite{C} that have been mentioned above.
Another idea characterizing Baire 1, real-valued functions, has been presented in \cite{EFKNPP}. 
Last but not least,  a game that characterizes Baire class 1 functions between arbitrary separable metrizable spaces 
has been defined by Notaro in very recent paper \cite{LN}.

Our first aim in this paper is to extend the result by Kiss. We simplify the proof of a harder implication of his result
by the use of $\varepsilon$-$\delta$ characterization of Baire 1 functions. Then we modify the game defined by Kiss
in two other manners, 
one in which Player II plays points in a space, and another in which Player II plays sets. 
	Whereas in the earlier versions
	of the game, considered by Kiss, Player II was playing in a space containing the range of a function, 
	here we let Player II play in the
	domain. This allows us to give strong game-theoretical characterizations 
	 of equi-Baire 1 families with both
	 a point-based and a set-based game, and finally, characterizations of Baire-measurable
	 and Lebesgue-measurable functions with set-based games.
 
We will use the following reasoning scheme throughout this work. 
\begin{lemma}\label{lem}
Let $G(f)$ be a game with a parameter function $f\in Y^X$. For a given class of functions $\FF\subset Y^X$ assume that:
\begin{enumerate}
	\item if $f\in\FF$ then Player II has a winning strategy in the game $G(f)$, and
	\item
	if $f\not\in \FF$ then Player I has a winning strategy in $G(f)$.
\end{enumerate}
Then the game $G(f)$ is determined and the class $\FF$ can be characterized by $G_f$:
\begin{enumerate}
	\item[(1')]
	$f\in\FF$ if and only if Player II has a winning strategy in the game $G(f)$, and
	\item[(2')] $f\not\in \FF$ if and only if  Player I has a winning strategy in $G(f)$.
\end{enumerate}
\end{lemma}

Assume that $X$ and $Y$ are Polish spaces. Through the paper, we assume that $d_X$ and $d_Y$ are the respective metrics in $X$ and $Y$.

Let us state preliminary facts on Baire 1 functions.
A function $f\colon X\to Y$ between Polish spaces $X, Y$ is called \emph{Baire class 1} whenever the preimage
$f^{-1}[U]$ is $F_\sigma$ in $X$ for any open set $U$ in $Y$. If $Y=\bR$, this is equivalent to the property that $f$
is the limit of a pointwise convergent sequence of continuous functions, see e.g. \cite[Theorem 24.10]{Ke}.

In the literature, we encounter various 
conditions which characterize the class of Baire 1 functions. The classical characterization given by Baire says that
$f$ is Baire 1 if and only if $f\restriction P$ has a point of continuity for every non-empty closed set $P\subseteq X$. 
This is the so-called \emph{Pointwise Continuity Property}, in short (PCP), see e.g.~\cite{B}.
An $\varepsilon$-$\delta$ characterization of Baire 1 functions, obtained in~\cite{LTZ}, says the following.
A function $f\colon X\to Y$ is Baire 1 whenever, for any positive number $\varepsilon$, there is a positive function $\delta_\varepsilon\colon X\to\bR$
such that for any $x_0,x_1\in X$,
\begin{equation} \label{star}
d_X(x_0,x_1)<\min\left\{\delta_\varepsilon(x_0),\delta_\varepsilon(x_1)\right\} \textup{\ implies\ }  d_Y(f(x_0),f(x_1))<\varepsilon.
\end{equation}
We will call such a $\delta_\varepsilon$ an \emph{$\varepsilon$-gauge for $f$}.

We say that a family $\F\subseteq Y^X$
is \emph{equi-continuous at a point $x\in X$} whenever
\begin{equation}\label{star1}
	\forall_{\varepsilon>0}\;\exists_{\delta>0}\;\forall_{f\in\F}\; \left(d_X\left(x,x_0\right)<\delta\Rightarrow 
	d_Y\left(f(x),f\left(x_0\right)\right)<\varepsilon\right).
	\end{equation}
$\F$ is equi-continuous if it is equi-continuous at every $x\in X$.

A family $\F\subseteq Y^X$ is said to fulfil the \emph{Point of Equicontinuity Property} ($\F$ has $(\textrm{PECP})$, in short)
if for every non-empty closed $P\subseteq X$, the family
$$\F\restriction P:=\{f\restriction P\colon f\in\F\}$$
has a point of equicontinuity.

We say that a family $\F\subseteq Y^X$ is \emph{equi-Baire 1}
if for any positive number $\varepsilon$ there is a positive function $\delta_\varepsilon\colon X\to\bR_+$
such that for any $x_0,x_1\in X$ and $f\in\mathcal{F}$ the condition (\ref{star}) holds
(i.e. all $f\in\FF$ have a family of common $\varepsilon$-gauges).
Clearly,
every equi-continuous family is equi-Baire 1 and has $\textrm{(PECP)}$, and the opposite implications do not hold.
(In fact, if $\FF$ is an equi-continuous family and $\varepsilon>0$ then there is $\delta>0$ which satisfies condition (\ref{star1}). 
Then the constant function $\delta_\varepsilon:=\delta$ satisfies (\ref{star}). 
On the other hand, if $f\in\bR^\bR$ is Baire 1 function that is not continuous, 
then the family $\{ f\}$ is equi-Baire 1 but not equi-continuous, see \cite{A}.)

Both definitions were introduced by D.~Lecomte in \cite{Le}. He proved the following equivalence.

\begin{theorem}[\hbox{\cite[Prop.~32]{Le}}]\label{thm:lecomte}
$\F$ has $(\textrm{PECP})$ if and only if
$\F$ is equi-Baire 1.
\end{theorem}
Let us mention that, since $X$ is a Polish space, 
a non-empty closed set in conditions (PCP) and (PECP) can be equivalently replaced
by a perfect set
(that is, a non-empty closed set without isolated points).

Note that the definition of equi-Baire 1 family of functions
was rediscovered later by A.~Alikhani-Koopaei in~\cite{A}.
The definition of families with (PECP) was used by 
E.~Glasner and M.~Megrelishvili in the context of dynamical systems
in \cite{GM} (under the name ``barely continuous family'').

Through the paper we assume the Axiom of Choice AC.

\section{Game characterizations of Baire 1 functions}

Recall the game defined by Kiss \cite{Ki}. 
Let $X$ and $Y$ be Polish spaces.
Let $f\colon X\to Y$ be an arbitrary function. At the $n$th step of the game $G_f$,
Player~I plays $x_n$, then Player~II plays $y_n$,
$$\begin{array}{lllllllll}
\textrm{Player I}  & x_0 &     & x_1 &     & x_2 &     & \cdots &        \\
\textrm{Player II} &     & y_0 &     & y_1 &     & y_2 &        & \cdots
\end{array}$$
with the rules that for each $n\in\mathbb{N}$:
\begin{itemize}
\item $x_n\in X$ and $d_X(x_n,x_{n+1})\leq 2^{-n}$;
\item $y_n\in Y$.
\end{itemize}
Since $X$ is complete, $x_n\to x$ for some $x\in X$. Player II wins if and only if $\la y_n\ra_{n\in\mathbb{N}}$ is convergent
and $y_n\to f(x)$.
Recall the main result of Kiss:
\begin{theorem}[\hbox{\cite[Theorem 1]{Ki}}]\label{thm:kiss}
The game $G_f$ is determined, and
\begin{itemize}
\item Player I has a winning strategy in $G_f$ if and only if $f$ is not of Baire class~1.
\item Player II has a winning strategy in $G_f$ if and only if $f$ is of Baire class~1.
\end{itemize}
\end{theorem}

The longest part of the original proof is the implication~(1) from Lemma~\ref{lem}:
``if $f$ is of Baire class 1 then Player II has a winning strategy''.
We show that it can be significantly shortened by the use of $\varepsilon$-$\delta$ characterization 
of Baire 1 functions.
We describe it in Lemma~\ref{le:G-strategy-II} which will be preceded by the following fact.

\begin{lemma}\label{sublemma}
A function	$f\colon X\to Y$ is Baire 1 if and only if it possesses a family of gauges $\{\delta_\varepsilon\colon \varepsilon>0\}$ 
such that for every $x\in X$ the map $\varepsilon\mapsto\delta_\varepsilon(x)$ is non-decreasing.
\end{lemma}
\begin{proof}
Only the implication ``$\Rightarrow$'' has to be proved. Assume that $f$ is Baire 1 and $\{\delta_\varepsilon\colon\varepsilon>0\}$ is a family of gauges for $f$. 
For every $\varepsilon>0$ fix $N_\varepsilon\in\mathbb{N}$ such that $N_\varepsilon=1$ if $\varepsilon\ge 1$ 
and $\frac{1}{N_\varepsilon}\le \varepsilon<\frac{1}{N_\varepsilon-1}$ for $\varepsilon<1$.
For $x\in X$ define 
$$\delta'_\varepsilon(x):= 
	\min\left\{ \delta_{\frac{1}{n}}(x) \colon n\le N_\varepsilon \right\}.$$
Clearly, if  $\varepsilon\le\varepsilon_1$  then $N_\varepsilon\ge N_{\varepsilon_1}$, hence for any $x\in X$ we have $\delta'_\varepsilon(x)\le\delta'_{\varepsilon_1}(x)$. We will show that $\{\delta'_\varepsilon\colon\varepsilon>0\}$ is a family of gauges for $f$.
Indeed, assume that $d_X(x_0,x_1)<\min(\delta'_\varepsilon(x_0),\delta'_\varepsilon(x_1))$ for some $\varepsilon>0$ and  $x_0,x_1\in X$.
Then $\frac{1}{N_\varepsilon}\le \varepsilon$ (by definition of $N_\varepsilon$)
and  $\delta'_\varepsilon(x_i)\le \delta_{\frac{1}{N_\varepsilon}}(x_i)$ for $i=0,1$, so
$$d_X(x_0,x_1)<\min(\delta_{\frac{1}{N_\varepsilon}}(x_0),\delta_{\frac{1}{N_\varepsilon}}(x_1)).$$ 
Hence $$d_Y(f(x_0),f(x_1))<\frac{1}{N_\varepsilon}\le\varepsilon.$$
\end{proof}

\begin{lemma}\label{le:G-strategy-II}
Let $\Delta:=\{\delta_\varepsilon\colon\varepsilon>0\}$  be a family of positive functions from $X$ into $\bR$ such that, 
for every $x\in X$, the map $\varepsilon\mapsto\delta_\varepsilon(x)$ is non-decreasing.
Then  there is a function $\$'_\Delta\colon X^{<\omega}\to X$ such that
for every sequence $\la x_n\ra$ with $d_X(x_n,x_{n+1})\leq 2^{-n}$ for each $n$, 
and for every $\varepsilon>0$ there exists $N_\varepsilon\in\mathbb{N}$ with the property
\begin{itemize}
	\item 
	for every Baire 1 function $f\colon X\to Y$, if $\Delta$ is a family of $\varepsilon$-gauges for $f$, 
	then 
	$$\forall_{n>N_\varepsilon}\; d_Y\left(f\left(\$'_{\Delta}\left(x_0,x_1,\ldots,x_n\right)\right), f\left(\lim_{n\to\infty}x_n\right)\right)<\varepsilon.$$
\end{itemize}
	In particular, $$\lim_{n\to\infty}f\left(\$'_{\Delta}\left(x_0,x_1,\ldots,x_n\right)\right)=f(\lim_{n\to\infty}x_n),$$ 
	so the function $\$:=f\circ\$'_\Delta$ is a winning strategy for Player II in the game $G_f$.
	\end{lemma}
	\begin{proof}
Fix $\varepsilon>0$ and a sequence $\la x_n\ra\subseteq X$ such that $d_X(x_n,x_{n+1})\leq 2^{-n}$ for each $n$. 
We may assume that, for each $n\in\mathbb{N}$, Player I plays $x_n$ in the $n$th move of the game $G_f$. For each $n\in\mathbb{N}$ let 
$K_n:=\overline{B}(x_n,2^{-n+1})$ be the closed ball around $x_n$. 
Note that this is the smallest closed ball around $x_n$ 
which ensures that $x:=\lim_{j\to\infty}x_j\in K_n$. 
Denote by $M_n$ the greatest index $m<n$ for which there exists a point
$x'\in K_n$ such that $K_n\subseteq B\left(x',\delta_{\frac{1}{m}}\left({x'}\right)\right)$; 
then pick one of them and call it $x'_n$. If such an index $m$ does not exist, put $M_{n}:=-\infty$.
Define
$$\$'_{\Delta}(x_0,\ldots,x_n):=\left\{\begin{array}{ll}
{x'_{n}}& \textup{if\ } M_n>-\infty, \\
x_n & \textup{otherwise}.
\end{array}\right.$$

It is enough to show that $\lim_{n\to\infty} f(\$'_\Delta(x_0,\ldots,x_n)) = f(x)$
for each Baire 1 function $f\colon X\to Y$ with the family of $\varepsilon$-gauges equal to $\Delta$.
Fix $\varepsilon>0$ and find a positive integer $M$ such that $1/M<\varepsilon$. 
There exists $N\in\mathbb{N}$ such that
for each $n\geq N$,
\begin{equation}\label{le:G-strategy-II:1}
x\in K_n\subseteq  B\left(x,\delta_{\frac{1}{M}}\left(x\right)\right).
\end{equation}
Since $x\in K_n$ for all $n$, it follows that $M_n\ge M> -\infty$ for all $n > \max\{N,M\}$.
Then $\delta_{\frac{1}{M_n}}(x')\le\delta_{\frac{1}{M}}(x')$ for all $x'\in X$,
hence
\begin{equation}\label{le:G-strategy-II:2}
\$'_\Delta(x_0,\ldots,x_n)=x_n' \in K_n\subseteq   {B\left(x'_n,\delta_{\frac{1}{M_n}}(x'_n)\right)\subseteq }    B\left(x'_n,\delta_{\frac{1}{M}}(x'_n)\right).
\end{equation}
From (\ref{le:G-strategy-II:1}) and (\ref{le:G-strategy-II:2}) we get
$$x,x'_n\in K_n\subseteq B\left(x,\delta_{\frac{1}{M}}\left(x\right)\right) \cap B\left(x'_n,\delta_{\frac{1}{M}}(x'_n)\right),$$
and so
$$d_X(x'_{n},x)<\min\left\{\delta_{\frac{1}{M}}\left(x'_{n}\right),\delta_{\frac{1}{M}}\left(x\right)\right\}.$$
To finish the proof it is enough to observe that, since $\delta_{\frac{1}{M}}$ is an $\frac{1}{M}$-gauge for $f$, so
$$d_Y\left(f\left(\$'_\Delta\left(x_0,\ldots,x_n\right)\right),f\left(x\right)\right)<\frac{1}{M}<\varepsilon.$$
\end{proof}

\begin{remark}
In the original proof, Kiss noted that ``the idea of the proof is to pick $y_n$ as the image of a point
in $\overline{B}(x_n,2^{-n+1})$ at which $f$ behaves {\it badly}''. In fact, we are able to shorten his argument,
since the family of $\varepsilon$-gauges encodes the ``bad'' behaviour of $f$. 
\end{remark}

\begin{remark}
From Lemma~\ref{le:G-strategy-II} (see also \cite[Theorem 1]{Ki}) it follows that Player~II has a winning strategy 
	in the game $G_f$ if and only if he
	has a winning strategy of the form $\$(x_0,x_1,\ldots,x_n)=f(\$'(x_0,x_1,\ldots,x_n))$.
This is a motivation for introducing the games $G'_f$ and $G''_f$.
\end{remark}

\subsection{The games $G'_f$ and $G''_f$}

Let $X$ and $Y$ be Polish spaces, $f\colon X\to Y$ be an arbitrary function. At the $n$th step of the game $G'_f$,
Player~I plays $x_n$, then Player~II plays $x'_n$,
$$\begin{array}{lllllllll}
\textrm{Player I}  & x_0 &     & x_1 &     & x_2 &     & \cdots &        \\
\textrm{Player II} &     & x'_0 &     & x'_1 &     & x'_2 &        & \cdots
\end{array}$$
with the rules that for each $n\in\mathbb{N}$:
\begin{itemize}
\item $x_n\in X$ and $d_X(x_n,x_{n+1})\leq 2^{-n}$;
\item $x'_n\in X$.
\end{itemize}
Since $X$ is complete, $x_n\to x$ for some $x\in X$. Player II wins if  
$\langle f(x'_n)\rangle $ is convergent to $f(x)$.
Otherwise, Player I wins.

As a consequence of Lemma~\ref{le:G-strategy-II}, we obtain the following result.
\begin{theorem}
The game $G'_f$ is determined, and
\begin{itemize}\label{thm:g-prim}
\item Player I has a winning strategy in $G'_f$ if and only if $f$ is not of Baire class~1.
\item Player II has a winning strategy in $G'_f$ if and only if $f$ is of Baire class~1.
\end{itemize}
\end{theorem}

\begin{proof}
We will apply Lemma~\ref{lem}, hence it is enough to show that:
\begin{enumerate}[(i)]
\item if $f$ is Baire class 1 then Player II has a winning strategy, and\label{thm:g-prim:i}
\item if $f$ is not of Baire class 1 then Player I has a winning strategy.\label{thm:g-prim:ii}
\end{enumerate}
To prove (\ref{thm:g-prim:i}) observe that the function $\$'_\Delta$ from Lemma~\ref{le:G-strategy-II}, 
for $\Delta$ being a family of gauges of $f$, is a winning strategy for Player II.

To see~(\ref{thm:g-prim:ii}), observe that the winning strategy for Player I in $G_f$ is also
a winning strategy for him in $G'_f$. Thus~(\ref{thm:g-prim:ii}) follows from Theorem~\ref{thm:kiss}.
\end{proof}

Now, we  propose a further modification of the game  to obtain a similar effect. This time, we will define a point-open game $G''_f$.
Let $X$ and $Y$ be Polish spaces, $f\colon X\to Y$ be an arbitrary function. 
At the first step of the game $G''_f$, Player I plays $x_0\in X$
and then Player II plays an open set $U_0\ni x_0$.
At the $n$th step of the game $G''_f$ ($n>0$),
Player I plays $x_n\in U_{n-1}$, then Player II plays an open set $U_n\ni x_n$:
$$\begin{array}{lllllllll}
\textrm{Player I}  & x_0 &      & x_1 &      & x_2 &      & \cdots &        \\
\textrm{Player II} &     & U_0 &     & U_1 &     & U_2 &        & \cdots
\end{array}$$
with the rules that for each $n\in\mathbb{N}$:
\begin{itemize}
\item $x_0\in X$, and $x_n\in U_{n-1}$ for $n>0$;
\item $U_n\ni x_n$.
\end{itemize}
If $\la x_n\ra $ is convergent and $\lim_{n\to\infty} f(x_n)=f(\lim_{n\to\infty}x_n)$ then Player II wins.
Otherwise, Player I wins.

\begin{theorem}\label{druga}
The game $G''_f$ is determined, and
\begin{itemize}
\item Player I has a winning strategy in $G''_f$ if and only if $f$ is not of Baire class~1.
\item Player II has a winning strategy in $G''_f$ if and only if $f$ is of Baire class~1.
\end{itemize}
\end{theorem}

\begin{proof}
It is enough to prove implications~(1) and~(2) from Lemma~\ref{lem}.

To see the first implication, assume that $f$ is Baire 1 and 
 let $\{\delta_\varepsilon\colon\varepsilon>0\}$ be a family of $\varepsilon$-gauges for $f$.
Without loss of generality (see Lemma~\ref{sublemma}) we may assume 
that, for any fixed $x\in X$, 
the sequence $\la\delta_{\frac{1}{n}}(x)\ra$ is decreasing,  and
$\delta_{\frac{1}{n}}(x)< 2^{-n}$ for every $n>0$.
In the $n$th move, Player~II plays
$U_n := B(x_n,\delta_{\frac{1}{n}}(x_n)/2)$.
This is a winning strategy for Player~II. Indeed, since  
$x_{n+1}\in U_n$, so $d_X(x_n,x_{n+1})<\diam(U_n)\le 2^{-n}$ for every $n$. 
Hence $\langle x_n\rangle$ is a Cauchy sequence in a complete space $X$, so it converges. 
Let $x:=\lim_{n\to\infty}x_n$. Fix $\varepsilon>0$ and $N>1/\varepsilon$.
Then $x\in B(x_n,\delta_{\frac{1}{N}}(x_n))$ for each $n\geq N$, and for all $n$ with $d_X(x,x_n)<\delta_{\frac{1}{N}}(x)$,
we have $d_Y(f(x),f(x_n))<\varepsilon$. Thus $\la f(x_n)\ra$ is convergent to $f(x)$.

Now assume that $f$ is not Baire class 1. 
Then there are a perfect set $P\subseteq X$, $y_0\in Y$ and $\varepsilon>0$ such that both sets 
$A:=\{x\in P\colon d_Y(f(x),y_0)<\varepsilon\}$ 
and $B:=\{ x\in P\colon d_Y(f(x),y_0)>2\varepsilon\}$ are dense in $P$.
The winning strategy for Player~I in the game $G''_f$ consists in choosing $x_n\in A$ for odd $n$  
and $x_n\in B$ for even $n$. 
In fact, if Player~I plays this strategy then the sequence $\la f(x_n)\ra$ is not a Cauchy sequence.
\end{proof}

\section{Games for equi-Baire 1 families of functions}
\label{sec:equi}

In this section, we modify games $G'_f$ and $G''_f$ to obtain characterizations of equi-Baire 1 families of functions.

Let $X$ and $Y$ be Polish spaces, let $\mathcal{F}\subseteq Y^X$. 
At the $n$th step of the game $G'_\mathcal{F}$,
Player I plays $x_n$, then Player II plays $x'_n$,
$$\begin{array}{lllllllll}
\textrm{Player I}  & x_0 &     & x_1 &     & x_2 &     & \cdots &        \\
\textrm{Player II} &     & x'_0 &     & x'_1 &     & x'_2 &        & \cdots
\end{array}$$
with the rules that for each $n\in\mathbb{N}$:
\begin{itemize}
\item $x_n\in X$ and $d_X(x_n,x_{n+1})\leq 2^{-n}$;
\item $x'_n\in X$.
\end{itemize}
Since $X$ is complete, $x_n\to x$ for some $x\in X$. 
Player II wins if
\begin{equation} \label {equi-conv}
\forall_{\varepsilon>0}\; \exists_{N\in\mathbb{N}}\;\forall_{n\geq N}\;\forall_{f\in\mathcal{F}}\;\; d_Y\left(f\left(x'_n\right),f\left(x\right)\right)<\varepsilon.
\end{equation}
(Then we say that~the indexed family of sequences $\{\la f(x'_n)\ra \colon f\in\mathcal F\}$ is \textit{equi-convergent}
to the indexed family 
$\{f(x)\colon f\in\mathcal F\}$).
Otherwise, Player I wins.

We will use the fact that (\ref{equi-conv}) implies the following Cauchy-type condition.
(A proof of this fact is left to the reader.)
\begin{equation} \label{Cauchy}
\forall_{\varepsilon>0}\;\exists_{N\in\mathbb{N}}\;\forall_{n,m\geq N}\;\forall_{f\in\mathcal{F}}\;\; d_Y\left(f\left(x'_n\right),f\left(x'_m\right)\right)<\varepsilon .
\end{equation} 
	 
\begin{theorem} \label{pierwsza}
The game $G'_\mathcal{F}$ is determined, and
\begin{itemize}\label{thm:gF-prim}
\item Player I has a winning strategy in $G'_\mathcal{F}$ if and only if $\mathcal{F}$ is not 
equi-Baire~1.
\item Player II has a winning strategy in $G'_\mathcal{F}$ if and only if $\mathcal{F}$ is equi-Baire~1.
\end{itemize}
\end{theorem}

\begin{proof}
We use the scheme of Lemma~\ref{lem}, so it is enough to show that:
\begin{enumerate}[(i)]
\item if $\mathcal{F}$ is equi-Baire 1 then Player II has a winning strategy, and\label{thm:gF-prim:i}
\item if $\mathcal{F}$ is not of equi-Baire 1 then Player I has a winning strategy.\label{thm:gF-prim:ii}
\end{enumerate}
To prove~(\ref{thm:gF-prim:i}) assume that $\FF$ is equi-Baire 1, fix a family $\Delta:=\{\delta_\varepsilon\colon\varepsilon>0\}$ of positive functions from $X$ into $\bR$ such that, 
	for every $x\in X$, the map $\varepsilon\mapsto\delta_\varepsilon(x)$ is non-decreasing being the family of common $\varepsilon$-gauges for $\FF$,  
the sequence $\la\delta_{\frac{1}{n}}(x)\ra$ is decreasing,  and
$\delta_{\frac{1}{n}}(x)< 2^{-n}$ for every $n>0$.
Then from Lemma~\ref{le:G-strategy-II}  used for $\Delta$ we obtain a function $\$'\colon X^{<\omega}\to X$ which is a winning strategy 
for Player II in the game $G'_\FF$.

To see~(\ref{thm:gF-prim:ii}) note that, if $\mathcal{F}$ is not equi-Baire 1, then there exists a non-empty
perfect set $P\subseteq X$ such that $\F\restriction P:=\{ f\restriction P\colon f\in\F\}$ has no point of equicontinuity (see Theorem~\ref{thm:lecomte}).
Since $P$ is Polish, to increase the readability we can assume, without loss of generality,  that $X=P$.

For any non-empty set $U\subseteq P$ let
$$\omega_{\F}(U):=\sup\left\{d_Y\left(f\left(u\right),f\left(v\right)\right)\colon u,v\in U\textup{\ and\ }  f\in\mathcal{F}\right\}.$$
For any $x\in P$ define the \emph{equi-oscillation of $\mathcal{F}$ at $x$}
as
$$\omega_{\F}(x):=\inf\left\{\omega_{\F}(B(x,h)\cap P)\colon h>0\right\}.$$
It is easy to observe that, for any integer $n>0$, the set
$$P_n:=\left\{x\in P\colon \omega_{\F}(x)\geq\frac{1}{n}\right\}$$
is closed,
and $x$ is a point of equicontinuity for $\mathcal{F}$ if and only if $\omega_{\mathcal{F}}(x)=0$.
Since $\F$ has no point of equicontinuity,
$\bigcup_{n\in\bN}P_n=P$. 
Since $P$ is a Polish space and all $P_n$'s are closed,
by the Baire Category Theorem there exists $P_n$ with non-empty interior.
Thus, without loss of generality, we may assume that for some $\varepsilon>0$, $\omega_{\mathcal{F}}(x)\geq\varepsilon$ for each $x\in P$.

We are ready to provide a strategy $\$$ for Player I.
In the first move he
picks $x_0\in P$. 
For $n>0$, in the $n$th move Player I takes $\$(x'_0,\ldots,x'_{n-1})=x_n\in P$ with 
$$x_n:=\left\{\begin{array}{ll}
x_{n-1} & \textup{if there exists\ } f\in\mathcal{F} \textup{\ such that\ } d_Y(f(x'_{n-1}), f(x_{n-1}))\geq\varepsilon/3;\\
a & \textup{otherwise,}
\end{array}\right.$$
where $a\in P$ and
\begin{itemize}
\item[(j)] $d_X(x_{n-1},a)<1/2^n$;\label{thm:gF-prim:ii:i}
\item[(jj)] $d_Y(f(x'_{n-1}), f(a))\geq 2\varepsilon/3$ for some $f\in\mathcal{F}$.\label{thm:gF-prim:ii:ii}
\end{itemize}
Such a choice is possible because the $P$ is dense-in-itself, $x_{n-1}\in P$, and $\omega_{\mathcal{F}}(x_{n-1})\geq\varepsilon$.

If the family of sequences $\{ \la f(x'_n)\ra\colon f\in\FF\}$ is not equi-convergent then Player~I wins. 
Otherwise, we use (\ref{Cauchy}). So, there exists $N\in\bN$ such that for all $n,m\geq N$ and all $f\in\mathcal{F}$,
\begin{equation} \label{star:x}
d_Y\left(f\left(x'_n\right),f\left(x'_m\right)\right)<\frac{\varepsilon}{6}.
\end{equation}
We claim that there exists $M\geq N$ with $x_n=x_M$ for all $n\geq M$.

We have two possibilities: either $x_m=x_{m-1}$ for all $m>N$, or
there exists $m_1>N$ such that $x_{m_1}\not=x_{m_1-1}$.
Since in the first case we are done, we assume the second one.
Then, 
by the formula defining $x_n$,
$$d_Y\left(f\left(x'_{m_1-1}\right), f\left(x_{m_1-1}\right)\right)<\frac{\varepsilon}{3} \textup{\ for each\ } f\in\mathcal{F}.$$
It follows from (jj) that 
\begin{equation} \label{star:xx}
d_Y\left(f_1\left(x'_{m_1-1}\right), f_1\left(x_{m_1}\right)\right)\geq \frac{2\varepsilon}{3}
\textup{\ for some\ } f_1\in\mathcal{F}.
\end{equation}
Thus, by $(\ref{star:x})$ and $(\ref{star:xx})$, for all $m\geq N$ we have
$$d_Y\left(f_1\left(x'_m\right), f_1\left(x_{m_1}\right)\right)
\geq
d_Y\left(f_1\left(x'_{m_1-1}\right), f_1\left(x_{m_1}\right)\right)
-
d_Y\left(f_1\left(x'_{m}\right), f_1\left(x'_{m_1-1}\right)\right)$$
$$\geq\frac{2\varepsilon}{3}-\frac{\varepsilon}{6}
=\frac{\varepsilon}{2}
\geq\frac{\varepsilon}{3}.$$
Hence, by the definition of $x_n$ for $n=m_1+1$ we obtain the equality $x_n=x_{m_1}$, so $f_{1}(x_n)=f_{1}(x_{m_1})$.
Therefore, $d_Y\left(f_1\left(x'_{n+1}\right), f_1\left(x_{n}\right)\right)\geq\frac{\varepsilon}{3}$, 
so by the definition of $x_{n+1}$ we get $x_{n+1}=x_n=x_{m_1}$. 
In this way we show, by induction, that $x_n=x_{m_1}$ for all $n\ge m_1$.
This finishes the proof of the claim.

Since  the sequence constructed by Player I is eventually constant, i.e.~$x_m=x_M$ for all $m\geq M$,
so $\lim_{n\to\infty}x_n=x_M$. 
Recall that in both variants of the formula defining $x_n$,
$$\forall_{n>0}\;\exists_{f\in\mathcal{F}}\;\; d_Y\left(f\left(x'_n\right), f\left(x_n\right)\right)\geq\frac{\varepsilon}{3}.$$
Therefore, since $x_m=x_M=\lim_{n\to\infty}x_n$ for all $m\geq M$,
$$\forall_{m\geq M}\; \exists_{f\in\mathcal{F}}\;\; d_Y\left(f\left(x'_m\right),f\left(\lim_{n\to\infty}x_n\right)\right)\geq\frac{\varepsilon}{3}.$$ 
Fix $f_2\in\mathcal{F}$ with
$$d_Y\left(f_2\left(x'_M\right),f_2\left(\lim_{n\to\infty}x_n\right)\right)\geq\frac{\varepsilon}{3}.$$ 
By $(\ref{star:x})$, for every $n\ge N$ we have
$$d_Y\left(f_2\left(x'_n\right),f_2\left(\lim_{n\to\infty}x_n\right)\right)\geq\frac{\varepsilon}{6},$$ 
thus $\la f_2(x'_n)\ra$ does not converge to $f_2\left(\lim_{n\to\infty}x_n\right)$, so Player I wins.
\end{proof}

Now, we will describe the game $G''_\F$ which is a modification of $G''_f$ for equi-Baire 1 families.


Let $X$ and $Y$ be Polish spaces, let $\mathcal{F}\subseteq Y^X$. 
At the first step of the game $G''_\mathcal{F}$,
Player I plays $x_0\in X$
and then Player II plays an open set $U_0\ni x_0$.
At the $n$th step of the game $G''_\mathcal{F}$,
Player I plays $x_n\in U_{n-1}$, then Player II plays an open set $U_n\ni x_n$:
$$\begin{array}{lllllllll}
	\textrm{Player I}  & x_0 &      & x_1 &      & x_2 &      & \cdots &        \\
	\textrm{Player II} &     & U_0 &     & U_1 &     & U_2 &        & \cdots
\end{array}$$
with the rules that for each $n\in\mathbb{N}$:
\begin{itemize}
	\item $x_0\in X$, and $x_n\in U_{n-1}$ for $n>0$;
	\item $U_n\ni x_n$.
\end{itemize}
Player II wins if
the sequence $\la x_n\ra$ converges to some $x\in X$, and  
the indexed family $\{ \la f(x_n)\ra \colon f\in\mathcal{F}\}$ is equi-convergent to $\{ f(x)\colon f\in\mathcal{F}\}$.
Otherwise, Player I wins.

\begin{theorem} \label{equi}
The game $G''_\mathcal{F}$ is determined, and
\begin{itemize}
	\item Player I has a winning strategy in $G''_\mathcal{F}$ if and only if $\mathcal{F}$ is not  equi-Baire 1.
	\item Player II has a winning strategy in $G''_\mathcal{F}$ if and only if $\mathcal{F}$ is equi-Baire 1.
\end{itemize}
\end{theorem}

\begin{proof}
Firstly, we show that, if $\mathcal F$ is equi-Baire 1, then Player II has a winning strategy.
We follow proof of Theorem \ref{druga}.
Let $\mathcal F$ be equi-Baire 1 and $\Delta=\{ \delta_\varepsilon\colon \varepsilon>0\}$ be the family of common gauges for $\FF$. 
Then for every $x\in X$,  $\delta_\varepsilon(x)$ does not depend on 
$f\in\mathcal F$.  We may assume 
that, for any fixed $x\in X$, 
the sequence $\la\delta_{\frac{1}{n}}(x)\ra$ is decreasing,  and
$\delta_{\frac{1}{n}}(x)< 2^{-n}$ for every $n>0$.
So, we choose $U_n:=B(x_n,\delta_{\frac{1}{n}}(x_n)/2)$.
Then $x_n \to x$ and note that the index $N$ such that $d_Y(f(x),f(x_n))<\varepsilon$ for all $n>N$ does not depend on $f\in\mathcal F$.
Hence the family of sequences $\{ \la f(x_n)\ra \colon f\in\mathcal F\}$ is equi-convergent 
to $\{f(x)\colon f\in \mathcal F\}$ and we are done.

Secondly, assuming that $\mathcal F$ is not equi-Baire~1, we will show that Player I has a winning strategy.
We follow the respective part in the proof of Theorem~\ref{pierwsza}. We can assume that there exist
a perfect set $P\subseteq X$ and $\varepsilon >0$ such that $\omega_{\mathcal{F}\restriction P}(x)\ge\varepsilon$ for each $x\in P$. Initially, Player I picks $x_0\in P$.
Let $n>0$. Since $x_{n-1}\in P$, we have $\omega_{\mathcal{F}\restriction P}(x_{n-1})\ge\varepsilon$. Thus, knowing that $U_{n-1}$ is an open neighbourhood of $x_{n-1}$, Player I can choose $x_n\in U_{n-1}$ and $f_n\in\mathcal F$ such that 
$$
d_Y({f_n}(x_{n-1}),f_n(x_n))\ge\frac{\varepsilon}{3}.
$$
This, by condition (\ref{Cauchy}), shows that
the family of sequences $\{\langle f(x_n)\rangle\colon f\in\mathcal F\}$ is not equi-convergent.
So, we have a winning strategy for Player~I.
\end{proof}

\section{Game characterization of measurable functions}

In this section, we  propose another modification of the game $G''_f$  to obtain characterizations
of Baire measurable and Lebesgue measurable functions. 

Let $\Sigma$ be a $\sigma$-algebra of subsets of a set $Z\neq\emptyset$. A function $f\colon Z\to Y$, where $Y$ denotes a topological space,
is called {\em $\Sigma$-measurable} if the preimage $f^{-1}[U]$ of any open set $U$ in $Y$ belongs to $\Sigma$.
Note that if $Y$ is a separable metric space, then $f$ is measurable whenever the preimage $f^{-1}[B]\in\Sigma$ for any open ball $B\subset Y$.

Let $H(\Sigma)$ be the $\sigma$-ideal given by
$H(\Sigma):=\{ A\subseteq Z\colon \forall B\subseteq A,\; B\in\Sigma\}$. Denote $\Sigma^+:=\Sigma\setminus H(\Sigma)$.
We say that the $\sigma$-algebra $\Sigma$ satisfies {\em condition ccc} if every family of pairwise disjoint sets in 
$\Sigma^+$ is countable.

We will use the following lemma.

\begin{lemma} \label{nonmeas}
Let $(Y,d)$ be a separable metric space and $\Sigma$ be a $\sigma$-algebra of subsets of $Z$ that satisfies condition ccc.
A function $f\colon Z\to Y$ is not $\Sigma$-measurable if and only if there exist a set $W\in\Sigma^+$, a point $y\in Y$ and $\varepsilon >0$ such that the sets $\{ z\in W\colon d(f(x),y)<\varepsilon\}$ and
$\{ z\in W\colon d(f(x),y)\ge 2\varepsilon\}$ intersect every subset of $W$ that belongs to $\Sigma^+$.
\end{lemma}
\begin{proof} ``$\Leftarrow$'' Take the open ball $B:=B(y,\varepsilon)$ in $Y$. 
The assumed condition implies that $W\cap f^{-1}[B]\not\in\Sigma$, hence $f^{-1}[B]\not\in\Sigma$ and consequently, $f$ is not $\Sigma$-measurable. 

``$\Rightarrow$'' Assume that $f$ is not $\Sigma$-measurable. Then there exist $y\in Y$ and $\varepsilon>0$ such that $f^{-1}[B(y,\varepsilon)]\not\in\Sigma$. 
Thus $A:=\{ z\in Z\colon d(f(z),y)<\varepsilon\}\not\in\Sigma$ and $B:=\{ z\in Z\colon d(f(z),y)\ge\varepsilon\}\not\in\Sigma$. 
Let $\mathcal{A}$ (respectively, $\mathcal{B}$) be a maximal family of pairwise disjoint $\Sigma^+$-subsets of $A$ (respectively, $B$). 
By condition ccc, the sets $\bigcup\cA$ and $\bigcup\cB$ belong to $\Sigma$. 
Let $A_0:=A\setminus\bigcup\cA$, $B_0:=B\setminus\bigcup\cB$, and $V:=A_0\cup B_0$. 
Then $A_0,B_0\not\in\Sigma$, and $V=Z\setminus (\bigcup\cA\cup\bigcup\cB)$, hence  $V\in \Sigma$.
Notice that for every $C\subseteq V$, if $C\in\Sigma^+$ then $A\cap C\ne\emptyset\ne B\cap C$. 

Since the space $Y$ is separable, there exists a sequence of open balls $B(y_n,\varepsilon_n)$, $n\in\bN$, 
such that $\bigcup_{n\in\bN}B(y_n,\varepsilon_n)=B(y,\varepsilon)$ and $d(y,y_n)+2\varepsilon_n<\varepsilon$
for all $n$.
Since $V\cap A=A_0\not\in\Sigma$,
there is $m\in\bN$ such that the sets $C:=V\cap f^{-1}[B(y_m,\varepsilon_m)]$ and $D:=V\setminus f^{-1}[B(y_m,\varepsilon_m)]$ 
are not in $\Sigma$. Again, let $\cC$ and $\cD$ be maximal families of pairwise disjoint $\Sigma^+$-sets contained in $C$ and $D$, respectively. 
Define $C_0=C\setminus\bigcup\cC$, $D_0=D\setminus\cD$, and $W=C_0\cup D_0=V \setminus (\bigcup\cC\cup\bigcup\cD)$. 
Then for any $S\subset W$, if $S\in\Sigma^+$ then $S\cap C_0\ne\emptyset$ and $d(f(z),y_m)<\varepsilon_m$ for each $z\in S$. 
On the other hand, $S\subseteq V$, hence $S\cap B\ne \emptyset$ and $d(f(z),y)\ge\varepsilon$ for $z\in S\cap B$, 
and then $d(f(z),y_m)\ge d(f(z),y)-d(y,y_m)\ge \varepsilon +2\varepsilon_m-\varepsilon = 2\varepsilon_m$.
\end{proof}

Let $X$ and $Y$ be topological Hausdorff spaces, $\Sigma$ be a $\sigma$-algebra on $X$, and $f\colon X\to Y$ be  an arbitrary function. 
We define the following game $G^\Sigma_f$.
At the first step of the game $G^\Sigma_f$, Player I plays $W\in\Sigma^+$, 
then Player II plays a set $U_0$.
At the $n$th step, 
where $n>0$,
Player I plays $x_n$ and Player II plays a $U_n$: 
$$\begin{array}{lllllllll}
	\textrm{Player I}  & W &      & x_1 &      & x_2 &      & \cdots &        \\
	\textrm{Player II} &     & U_0 &     & U_1 &     & U_2 &        & \cdots
\end{array}$$
with the rules that for each $n\in\mathbb{N}$:
\begin{itemize}
	\item  $x_n\in U_{n-1}$ for $n>0$;
	\item $U_n\in\Sigma^+$ and $U_n\subseteq W$.
\end{itemize}
Player II wins the game $G^\Sigma_f$ if 
$\la x_n\ra $ is convergent and $\lim_{n\to\infty} f(x_n)=f(\lim_{n\to\infty}x_n)$. Otherwise Player I wins.

\begin{lemma}\label{ostatni}
Assume that $X$ is a topological Hausdorff space and $\Sigma$ is a $\sigma$-algebra of subsets of $X$ that satisfies condition ccc. Let $(Y,d)$ be a separable metric space.
If $f\colon X\to Y$ is not $\Sigma$-measurable then Player I has a winning strategy in the game {$G^\Sigma_f$}.
\end{lemma}
\begin{proof}
Assume that $f$ is not $\Sigma$-measurable. 
By Lemma~\ref{nonmeas} there exist a set $W\subseteq X$, $y\in Y$ and $\varepsilon>0$ such that $W\in\Sigma^+$ and both sets 
$A:=\{ x\in W\colon d_Y(f(x),y)<\varepsilon\}$ and $B:=\{ x\in W\colon d_Y(f(x),y)\ge 2\varepsilon\}$ intersect
every subset of $W$ which is in $\Sigma^+$.
Let Player~I play the following strategy. 
At the first step, he
chooses the set $W\in\Sigma^+$ obtained above.
If $n>0$, then $U_{n-1}\subseteq W$ and he
chooses $x_n\in A\cap U_{n-1}$ when $n$ is even and $x_n\in B\cap U_{n-1}$ when $n$ is odd. 
Then $d(f(x_{2n}),y)\le\varepsilon$ and $d(f(x_{2n+1}),y)\ge 2\varepsilon$, 
so {$d(f(x_{2n}),f(x_{2n+1}))\ge\varepsilon$} for every $n$, and therefore $\la f(x_n) \ra$ is not convergent.	
\end{proof}
First, we will characterize Baire measurable functions,
 that is $\Sigma$-measurable functions, where $\Sigma=\Baire$ denotes
the $\sigma$-algebra of sets with the Baire property in a topological space, cf.~\cite[8.21]{Ke}. 
Note that the ideal $H(\Baire)$ is equal to the family of all meager sets,
cf.~\cite[Theorem 5.5]{JO} (recall that we assume AC) and if $X$ is second countable then the algebra $\Baire$ satisfies condition ccc, 
cf.~\cite[Theorem 7.5]{LB}.

\begin{theorem} \label{cat}
Let $X$ be a Polish space, $Y$ be a separable metric space, and $f\colon X\to Y$ be a function.
	Then the game $G^\Baire_f$ 
	is determined, and
	\begin{itemize}
		\item Player I has a winning strategy in $G^\Baire_f$ if and only if $f$ is not Baire measurable;
		\item Player II has a winning strategy in $G^\Baire_f$ if and only if $f$ is Baire measurable.
	\end{itemize}
\end{theorem}
\begin{proof}
We apply Lemma~\ref{lem}, thus we have to prove two implications:
\begin{enumerate}[(i)]
	\item
	if $f$ is Baire measurable then Player II has a winning strategy in the game $G^\Baire_f$;
	\item
	if $f$ is not Baire measurable then Player I has a winning strategy in $G^\Baire_f$.
\end{enumerate}

To prove (i) assume that $f$ is Baire measurable. We will describe a winning strategy for Player II in the game $G^\Baire_f$.
Let $G\subseteq X$ be a dense $G_\delta$ set such that \mbox{$f\restriction G$} is continuous. (See \cite[Theorem 8.38]{Ke}.)
Let $W\in\Baire^+$ 
	be chosen by Player I at the first move. Then Player II fixes any point $a\in G\cap W$ and picks $U_0:=B(a,1)\cap G\cap W$. At the $(n+1)$-th move, Player I chooses $x_{n+1}\in U_n$.
Then Player II plays $U_{n+1}:=\{x_{n+1}\}\cup (B(a,\frac{1}{n+1})\cap G\cap W)$. When the game is finished, one of the two cases is possible: 
either $x_n=x_N$ for some $N\in\bN$ and all $n>N$, or
for every $\varepsilon>0$ there is $N\in\bN$ such that $x_n\in B(a,\varepsilon)\cap G$ for all $n>N$,
 which implies $\lim_{n\to\infty}x_n=a$.
In both cases, $\lim_{n\to\infty}f(x_n)=f(\lim_{n\to\infty}x_n)$.

The implication~(ii) follows from Lemma~\ref{ostatni} with $\Sigma=\Baire$.
\end{proof}

A similar idea can be used to characterize Lebesgue measurable functions from $X:=\bR^k$ to a separable metric space $Y$. 
Let $\Leb$ denote the $\sigma$-algebra of Lebesgue measurable subsets of $\bR^k$. 
Note that the ideal $H(\Leb)$ consists exactly of Lebesgue null sets in $\bR^k$, cf. \cite[Theorem 5.5]{JO},
and $\Leb$ satisfies condition ccc, cf.~\cite[Theorem 7.5]{LB}.

\begin{theorem} \label{meas}
	Let $X=\bR^k$, $Y$ be a separable metric space, and $f\colon X\to Y$ be a function.
	Then the game $G^\Leb_f$ is determined, and
	\begin{itemize}
		\item Player I has a winning strategy in $G^\Leb_f$ if and only if $f$ is not measurable;
		\item Player II has a winning strategy in $G^\Leb_f$ if and only if $f$ is measurable.
	\end{itemize}
\end{theorem}
\begin{proof}
We use Lemma~\ref{lem}. First assume that $f$ is measurable. 
	We will describe a winning strategy for Player II.   
	Let $W\in\Leb^+$ 
		be chosen at the initial move by Player~I. 
	Let $F\subseteq W$ be a compact set with positive measure.
	By the Lusin theorem (applied to the space $F$ with the restricted Lebesgue measure),
	see \cite[Theorem 17.12]{Ke},
	there exists a closed set $F_0\subseteq {F}$ 
	such that $f\restriction F_0$ is continuous and the Lebesgue measure 
	of $F_0$ is finite and positive. Then Player~II picks $U_0:=F_0$. 
	At the $(n+1)$-th move, Player~I chooses $x_{n+1}\in U_n$.
	Then Player~II plays $U_{n+1}:=\{ x_{n+1}\}\cup F_{n+1}$ where $F_{n+1}\subseteq F_n$ is a closed set of a positive measure
	with the diameter less than $\frac{1}{n+1}$. When the game is finished, we have $\bigcap_{n\ge 0} F_n=\{ a\}$ 
	for some $a\in X$. As in the previous proof,
	we infer that $\la x_n\ra$ is eventually constant or
	 $x_n\to a\in F_0$,
	and so $\lim_{n\to\infty}f(x_n)=f(\lim_{n\to\infty}x_n)$.
	
	The second implication follows from Lemma~\ref{ostatni} with $\Sigma=\Leb$.
\end{proof}

\begin{remark}
By \cite[Theorem 17.12]{Ke}, see also \cite{St}, Theorem~\ref{meas} can be extended to the case when $X$ is a Polish space equipped with a $\sigma$-finite Borel regular measure.
\end{remark}

\subsection*{Acknowledgements}
	We would like to thank the referee for
	several useful remarks that have helped us to improve the former version of the
	paper.
 \subsection*{Compliance with Ethical Standards}
The authors declare that they have no conflict of interest, and add that there is no data
associated with this work.


\end{document}